\documentclass{article}

\usepackage{header}

\usepackage[backend=bibtex,style=numeric,doi=false,isbn=false,url=false,maxnames=10]{biblatex}
\AtBeginDocument{
}

\AtEndDocument{
  \printbibliography
  \authaddresses
}

\begin{document}

\title{\sc A note on the symplectic classification of almost-toric systems}
\author{Xiudi Tang}

\date{}

\maketitle

\begin{abstract}
  Since simple semitoric systems were classified about fifteen years ago, and semitoric systems five years ago, we want to move a step forward to almost-toric systems.
  We give a classification of compact almost-toric systems in dimension four up to fiber-preserving symplectomorphisms, in terms of the base, Taylor series, and twisting indices, analogous to the five invariants for semitoric systems.
  For convenience, we specify an ordering of focus-focus values and a choice of two cut rays at each of them.
\end{abstract}

\section{Introduction}
\label{sec:intro}

The classification problem is one of the fundamental subjects in integrable systems.
After the work by Atiyah~\cite{MR642416}, Guillemin--Sternberg~\cite{MR664117}, and Delzant~\cite{MR984900} on the correspondence between compact \emph{toric systems} and Delzant polytopes, Pelayo--V\~u Ng\d{o}c classified all simple \emph{semitoric systems} in dimension $4$ in~\cite{MR2534101, MR2784664}, and the simpleness condition was removed in Palmer--Pelayo--Tang~\cite{MR4797864}, using generalized Delzant polytopes.
We also mention Karshon--Lerman~\cite{MR3371718} where noncompact toric systems were classified.
All these classifications are up to symplectic equivalences, namely fiber-preserving symplectomorphisms.
As far as we know, such a classification for \emph{almost-toric systems} in a similar format has not been available.

Statements for the topological and symplectic equivalence classifications of integrable systems are found in Zung~\cite{MR2018823}.
The former one was realized in Leung--Symington~\cite{MR2670163} for compact almost-toric systems in dimension $4$ (actually, they did this for more manifolds than integrable systems). 
We feel it necessary to construct explicit invariants for almost-toric systems as \cite{MR2534101, MR2784664, MR4797864} did, write down and prove a detailed classification result up to symplectic equivalences.

In \cref{sec:ator-sys} we recall the definition of almost-toric systems.
In \cref{sec:affine-invariant} we give the recipe for the \emph{almost-toric invariants}.
In \cref{sec:classification} we prove that the map assigning the complete almost-toric ingredients to a semitoric system is bijective, so as to prove the classification of almost-toric systems.

\section{Almost-toric systems}
\label{sec:ator-sys}

\begin{definition} \label{def:integrable-system}
  An \emph{integrable system} $(M^{2n}, \om, \mu)$ is a symplectic manifold $(M, \om)$ with a Hamiltonian $\mathfrak{t}^n$-action such that the rank of the momentum map $\mu \colon M \to (\mathfrak{t}^n)^*$ equals $n$ almost everywhere.
\end{definition}

In other words, an integrable system is a symplectic manifold $(M^{2n}, \om)$ equipped with a smooth function $\mu = (f_1, \dotsc, f_n) \colon M \to \R^n \simeq (\mathfrak{t}^n)^*$ such that
\begin{itemize}
  \item the components $f_i$, $i \in \{1, \dotsc, n\}$, Poisson commute; and
  \item their derivatives $\der f_i$, $i \in \{1, \dotsc, n\}$, are linearly independent almost everywhere.
\end{itemize}

\begin{definition} \label{def:atoric-system}
  An integrable system $(M^4, \om, \mu = (J, H))$ is 
  \begin{itemize}
    \item \emph{almost-toric} if all singular points of $\mu$ are nondegenerate and without hyperbolic components, and moreover, $\mu$ is proper and has connected fibers;
    \item \emph{semitoric}, if it is almost-toric, $J$ is proper, and $X_J$ is $2\pi$-periodic;
    \item \emph{toric}, if both $X_J$ and $X_H$ are $2\pi$-periodic.
  \end{itemize}
\end{definition}

A critical point of $\mu$ in an almost-toric system has one of the following three types, by Eliasson's linearization theorem~\cite{Eliasson_1984, MR3098203}.
There are local symplectic coordinates $(x_1, \xi_1, x_2, \xi_2)$ for each type centered at that critical point and a local diffeomorphism $E \colon \R^2 \to \R^2$ such that
\begin{description}
  \item[focus-focus] $E \circ \mu = (x_1\xi_2-x_2\xi_1, x_1 \xi_1 +x_2 \xi_2)$;
  \item[elliptic-regular] $E \circ \mu =( \frac{x_1^2 + \xi_1^2}{2}, \xi_2)$;
  \item[elliptic-elliptic] $E \circ \mu = ( \frac{x_1^2 +\xi_1^2}{2}, \frac{x_2^2 + \xi_2^2}{2})$.
\end{description}

Let $\tcalM$ denote the collection of compact almost-toric integrable systems of dimension $4$.
Two systems $(M, \om, \mu), (M', \om' , \mu') \in \tcalM$ are \emph{symplectically equivalent} via a diffeomorphism $G \colon \mu(M) \to \mu'(M')$ if there is a symplectomorphism $\varphi \colon (M, \om) \to (M', \om')$ lifting $G$, in the sense that $\mu' \circ \varphi = G \circ \mu$.
Let $\calM$ denote the collection of symplectic equivalence classes via orientation-preserving diffeomorphisms of systems in $\tcalM$, which we would like to classify.

\section{The invariants}
\label{sec:affine-invariant}

In \cite[Proposition 3.5]{MR2670163}, the symplectic equivalence classification of almost-toric systems, is written as follows.
In dimension four, an almost-toric system is determined, up to fiber-preserving symplectomorphism, by its base $(B, \calA)$ and the local structure of the fibered neighborhoods of its nodal fibers.
In this section, we give detailed descriptions of the affine and semiglobal invariants.

\subsection{Almost-toric bases}
\label{ssec:ator-base}

Let $(M, \om, \mu) \in \tcalM$ and let $B = \mu(M)$.
Let $M_\rmf \subset M$ be the set of focus-focus singular points of $\mu$ and let $B_\rmf = \mu(M_\rmf) \subset B$.
The $i$-stratum $B^{(i)}$ of $B \setminus B_\rmf$, $i \in \Set{0, 1, 2}$, is the image of elliptic-elliptic, elliptic-regular, regular points, respectively, under $\mu$.

\begin{lemma} \label{lem:strat}
  The collection
  \begin{equation*}
    \calS = \Set{B^{(2)}, B^{(1)}, B^{(0)}, B_\rmf}
  \end{equation*}
  is a \emph{stratification} on $B$ in the sense that
  \begin{itemize}
    \item $\calS$ is a decomposition of $B = \bigsqcup_{X \in \calS} X$;
    \item every stratum is a smooth manifold of a non-negative integer dimension: $B^{(i)}$ has dimension $i$ and $B_\rmf$ has dimension $0$;
    \item $S$ satisfies the \emph{axiom of the frontier}: if $X, Y \in \calS$ and $Y \cap \overline{X} \neq \emptyset$, then $Y \subseteq \overline {X}$; if $Y \subseteq \overline {X}$ we denote $Y \leq X$ and this is partial order among strata which respects the order of dimensions: $B_\rmf < B^{(2)} > B^{(1)} > B^{(0)}$.
  \end{itemize}
\end{lemma}

Recall the \emph{standard lattice bundle} $\calA_0 = \Z \der x \oplus \Z \der y \subseteq T^* \R^2$.
An \emph{integral affine structure} on $B \setminus B_\rmf$ is a \emph{lattice} $\calA \subseteq T^* (B \setminus B_\rmf)$ that is locally isomorphic to $\calA_0$ with structure group $\AGL(2, \Z)$.

\begin{definition} \label{def:base}
  For $c \in B$ let $2\pi \Period_c \subset T^*_c\R^2$ be the isotropy subgroup of the action on $\mu^{-1}(c)$ by $T^*_c\R^2$, given by the time-$2\pi$ map of the flow of $-\om^{-1} \mu^*\beta$ for any $\beta \in T^*_c\R^2$.
  The abelian group $\Period_c$ is called the \emph{period lattice} of $(M, \om, \mu)$ over $c$
  Then $\calA = \bigcup_{c \in B} \Period_c$ is called the \emph{period bundle} of $(M, \om, \mu)$, which is an integral affine structure on $B \setminus B_\rmf$, and $(B, \calA, \calS)$ is called the \emph{base} of $(M, \om, \mu)$.
\end{definition}

As in~\cite{MR2722115}, $(B, \calA, \calS)$ is an integral affine manifold with corners and nodes.

\subsection{The Taylor series invariants}
\label{ssec:taylor-series}

Since $B_\rmf$ is discrete and finite, we fix an ordering $(c_i)_{i = 1}^\vf$ of $B_\rmf$.
We call the number $m_i \in \N$ of focus-focus singular points in $\mu^{-1} (c_i)$ the \emph{multiplicity} of $c_i$.
There is a choice to make to order these points.
For any smooth function $f$ in a neighborhood of $c_i$ in $B$, the flow of $X_f$ on $\mu^{-1}(c_i)$ is either periodic or traveling among the singular points in two possible orders reverse to each other.
We arrange them in one of the two cyclic orders $(p_{i, u})_{u \in \Z_{m_i}}$; for details see \cite[Section 3.1]{MR4694440}.
Then, the choice of numbering is unique up to cyclic permutation, which is why we take the index $u$ to be in $\Z_{m_i}= \Z/m_i\Z$.

By Eliasson's linearization theorem for non-degenerate focus-focus points~\cite{MR3098203}, there is a neighborhood $U_i$ of $c_i$ in $B$ and for any $u \in \Z_{m_i}$ there is a neighborhood $V_{i, u}$ of $p_{i, u}$ in $M$, a symplectomorphism $\psi_{i, u} \colon (V_{i, u}, \om) \to (\hat{V}_{i, u} \subseteq \R^4, \om_0)$ sending $p_{i, u}$ to $0$, and an orientation-preserving diffeomorphism $E_{i, u} \colon U_i \to \hat{U}_{i, u} \subseteq \R^2$ sending $c_i$ to $0$ such that $\psi_{i, u}$ lifts $E_{i, u}$.
In other words, we have $q \circ \psi_{i, u} = E_{i, u} \circ \mu$, where $q (x_1, \xi_1, x_2, \xi_2) = (x_1 \xi_2 - x_2 \xi_1, x_1 \xi_1 + x_2 \xi_2)$ is the local model of a focus-focus point in $\R^4$, as in \cref{sec:ator-sys}.

The choices of $E_{i, u}$ are not unique, but the one with $X_{\proj_2 \circ E_{i, u}}$ giving the preferred order of focus-focus points in unique up to a flat function; see \cite[Lemma 2.21]{MR4694440}.
Let $J_i = \proj_1 \circ E_{i, u}$ be the abscissa in the $E_{i, u}$-coordinates which is independent of $u \in \Z_{m_i}$, and we let $\ell_i$ be the level set of $J_i$ through $c_i$ in $U_i$.
A \emph{cut ray} $\epsilon_i$ at $c_i$ is the union of $\Set{c_i}$ with one of the two component of $\ell_i \setminus \Set{c_i}$.
In the $E_{i, u}$-coordinates, we could distinguish the two cut rays at $c_i$ by denoting by $\ell_i^+$ the one with greater ordinates and $\ell_i^-$ with less ordinates.

\begin{definition}
  A function $A_i = (A_i^1, A_i^2) \colon U_i \to \R^2$ is called a choice of \emph{local affine coordinates} if
  \begin{itemize}
    \item it is continuous and the restriction $\Res{A}_{U_i \setminus \ell_i^+}$ is smooth;
    \item the first component $A_i^1 = J_i$; in particular, $\der \Res{A_i^1}_c \in \Period_c$ for any $c \in U_i \setminus \Set{c_i}$;
    \item the differential of the second component $\der \Res{A_i^2}_c \in \Period_c$ for any $c \in U_i \setminus \ell_i^+$.
  \end{itemize}
\end{definition}

Identifying $\R^2\cong\C$, let $\ln_+ \colon \C \setminus \imag \R^+ \to \C$ be the determination of $\ln$ with $\ln_+ 1 = 0$ and branch cut at $\imag \R^+$, and let $K_+ \colon \C \setminus \imag \R^+ \to \R$ be given by $K_+(c)= -\Im (c \ln_+ c - c)$.
We define $\tldS_i \colon U_i \setminus \ell_i^+ \to \R$ by
\begin{equation*} \label{eq:def-tS}
  \tldS_i = 2\pi A_i^2 - \sum_{u \in \Z_{m_i}} E_{i, u}^* K_+.
\end{equation*}
By \cite[Lemma 3.5]{MR4694440}, $\tldS_i$ can be extended to a smooth function in $U_i$ (possibly shrunk), which we still denote by $\tldS_i \colon U_i \to \R$.

\begin{definition} \label{def:Taylor-series}
  Let $X = \der x, Y = \der y$.
  Performing a Taylor expansion of $\tldS_i$ around the origin under coordinates $E_{i, u}$, we get a power series
  \begin{equation*} 
    \tsfs_{i, u} = \Tl_0[\tldS_i \circ E_{i, u}^{-1}] = \sum_{p, q = 0}^ \infty \tsfs_{i, u}^{(p, q)} X^p Y^q,
  \end{equation*}
  called the \emph{relative action Taylor series at $p_{i, u}$}, and
  \begin{equation*} 
    \sfs_{i, u} = \tsfs_{i, u} + 2\pi X \Z
  \end{equation*}
  the \emph{action Taylor series at $p_{i, u}$}, where $u \in \Z_{m_i}$.
  Expanding the transition maps between coordinates $E_{i, u}$ and $E_{i, v}$, we get
  \begin{equation*} 
    \sfg_{i, u, v} = \Tl_0[\proj_2 \circ E_{i, v} \circ E_{i, u}^{-1}] = \sum_{p, q = 0}^ \infty \sfg_{i, u, v}^{(p, q)} X^p Y^q
  \end{equation*}
  the \emph{transition Taylor series} from $p_{i, u}$ to $p_{i, v}$, where $u, v \in \Z_{m_i}$.
\end{definition}

\begin{definition} \label{def:Taylor-series-label}
  Let $\sfl_i = [\sfs_{i, u}, \sfg_{i, u, v}]_{u, v \in \Z_{m_i}}$ denote the orbit of $(\sfs_{i, u}, \sfg_{i, u, v})_{u, v \in \Z_{m_i}}$ under the action of $\Z_{m_i}$ by $[w] \cdot (\sfs_{i, u}, \sfg_{i, u, v})_{u, v \in \Z_{m_i}} = (\sfs_{i, u + w}, \sfg_{i, u + w, v + w})_{u, v \in \Z_{m_i}}$ for $w \in \Z$.
  We call $\sfl_i$ the \emph{Taylor series invariant} at $c_i$.
\end{definition}

We note that although the action Taylor series and the transition Taylor series are independent of the choice of local affine coordinates, the relative action Taylor series do depend on them.
Since the Taylor series invariants depend on the choices of the ordering of focus-focus values and cut rays, we give the definition of marked systems.

\begin{definition}
  Let $(M, \om, \mu) \in \tcalM$ and let $B = \mu(M)$.
  Let $(c_i)_{i = 1}^\vf$ be an ordering of $B_\rmf$, and let $(\epsilon_i)_{i = 1}^\vf$ be a tuple of cut rays at each $c_i$.
  Then we call
  \begin{align*}
    \Pa{(M, \om, \mu), (c_i)_{i = 1}^\vf, (\epsilon_i)_{i = 1}^\vf}
  \end{align*}
  a compact \emph{marked almost-toric system} of dimension $4$.
  Let $\tcalM_0$ denote the collection of compact marked almost-toric systems of dimension $4$.
\end{definition}

\section{The classification}
\label{sec:classification}

The base of an almost-toric fibration has a geometric structure very similar to a toric base. 
Indeed, the only difference between an almost-toric and a toric base (on a local level) is the presence of nodes which in turn have neighborhoods whose affine structures (on the complements of the nodes) depend only on the multiplicity of the node.

\subsection{Almost-toric closed disks}
\label{ssec:ator-disk}

To specify the singular affine manifolds that could appear as bases of almost-toric manifolds, we first define some local models, similar to Euclidean spaces being local models of manifolds.
Let the \emph{standard toric base} be $(\R_{\geq0}^2, \calA_0, \calS_\rme)$ where $\calS_\rme = \Set{\R_{>0}^2, \R_{>0} \times \Set{0} \cup \Set{0} \times \R_{>0}, \Set{0}\!^2, \emptyset}$.
The \emph{standard focus-focus base} $(\R^2 \setminus \Set{0}, \calA_m, \calS_\rmf)$ of \emph{multiplicity} $m \in \N$ where the fiber $(\Period_m)_z$ of $\calA_m$ over $z = (x, y)$ is given by
\begin{align*}
  \Z \der x \oplus \Z \Pa{ -\frac{m}{2\pi} \Im \ln z \der x - \frac{m}{2\pi} \Re \ln z \der y}
\end{align*}
and $\calS_\rmf = \Set{\R^2 \setminus \Set{0}, \emptyset, \emptyset, \emptyset}$.

The following definition specializes \cite[Definition 5.1]{MR2024634}.

\begin{definition}
  An \emph{almost-toric closed disk} $(B, \calA, \calS)$ is a compact subset $B$ of $\R^2$ equipped with a stratification $\calS = \Set{B^{(2)}, B^{(1)}, B^{(0)}, B_\rmf}$, and an integral affine structure $\calA$  on $B \setminus B_\rmf$ such that
  \begin{itemize}
    \item the stratification $\calS$ satisfies \cref{lem:strat};
    \item each $c \in B \setminus B_\rmf$ has a neighborhood $U$ such that $(U, \calA)$ is affine isomorphic to a neighborhood of $(\R_{>0}^2, \calA_0)$ at a point in the stratum of the same dimension as $c$;
    \item each $c \in B_\rmf$ has a neighborhood $U$ such that $(U \setminus \Set{c}, \calA)$ is affine isomorphic to a neighborhood of the puncture in $(\R^2 \setminus \Set{0}, \calA_m, \calS_\rmf)$ for some $m \in N$.
  \end{itemize}
\end{definition}

For an almost-toric closed disk, we can also give an ordering $(c_i)_{i = 1}^\vf$ of $B_\rmf$.
Let $E_i \colon (U_i \setminus \Set{c_i}, \calA) \to (\R^2 \setminus \Set{0}, \calA_m)$ be an orientation-preserving affine embedding for $U_i$ a neighborhood of $c_i$ in $B$.
Let $J_i = \proj_1 \circ E_{i, u}$ be the abscissa, and we let $\ell_i$ be the level set of $J_i$ through $c_i$ in $U_i$, called the \emph{eigenline} (of the monodromy) at $c_i$.
A \emph{cut ray} $\epsilon_i$ at $c_i$ is one of $\ell_i^\pm$ where
\begin{align*}
  \ell_i^+ &= E_i^{-1} \Pa{\Set{0} \times [0, \infty)}, &\ell_i^- &= E_i^{-1} \Pa{\Set{0} \times (-\infty, 0]}.
\end{align*}
The \emph{multiplicity} of $c$ in $(B, \calA, \calS)$ equals $m$.
One can check that the above definitions for an almost-toric closed disk match those for an almost-toric system whose base is that disk.

\begin{definition}
  An \emph{isomorphism} $G$ from an almost-toric closed disk $(B, \calA, \calS)$ to $(B', \calA', \calS')$ is an orientation-preserving diffeomorphism $G \colon B \to B'$ that preserves the stratification, in the sense that $G(B^{(k)}) = (B')^{(k)}$ and $G(B_\rmf) = B'_\rmf$, and is an affine map in the sense that $G^* \calA' = \calA$.
\end{definition}

Combining the analysis of the local structures of bases of almost-toric systems and \cite[Proposition 3.4]{MR2670163}, \cite[Theorem 5.2]{MR2024634}, we write down the following conclusion.

\begin{theorem} 
  The base of any almost-toric system $(M, \om, \mu)$ is an almost-toric closed disk, and any almost-toric closed disk $(B, \calA, \calS)$ is the base of an almost-toric system.
\end{theorem}

\begin{definition}
  A \emph{marked almost-toric closed disk} is a tuple
  \begin{align*}
    \Pa{(B, \calA, \calS), (c_i)_{i = 1}^\vf, (\epsilon_i)_{i = 1}^\vf}
  \end{align*}
  where $(B, \calA, \calS)$ is an almost-toric closed disk with an ordering $(c_i)_{i = 1}^\vf$ of $B_\rmf$ and a tuple of cut rays $(\epsilon_i)_{i = 1}^\vf$.
  An \emph{isomorphism} $G$ from a marked almost-toric closed disk $\Pa{(B, \calA, \calS), (c_i)_{i = 1}^\vf, (\epsilon_i)_{i = 1}^\vf}$ to $\Pa{(B', \calA', \calS'), (c'_i)_{i = 1}^\vf, (\epsilon'_i)_{i = 1}^\vf}$ is an isomorphism from an almost-toric closed disk $(B, \calA, \calS)$ to $(B', \calA', \calS')$ such that $c'_i = G (c_i)$ and $\epsilon'_i = G(\epsilon_i)$ for each $i \in \Set{1, \dotsc, \vf}$.
\end{definition}

The orientation-preserving symplectic equivalence between elements in $\tcalM_0$, if exists, is unique, unless $\vf = 0$.

\subsection{The focus-focus labels}
\label{ssec:focus-label}

Let $\R[[X, Y]]$ denote the set of Taylor series in two variables with real coefficients and $\R_0[[X, Y]]$ be the subset of those which have zero constant terms. 

\begin{definition} \label{def:focus-label}
  A \emph{focus-focus label of multiplicity $m \in \N$} is the orbit $\sfl = [\sfs_u, \sfg_{u, v}]_{u, v \in \Z_m}$ of a tuple
  \begin{equation*}
    (\sfs_u, \sfg_{u, v})_{u, v \in \Z_m} \in (\R_0[[X, Y]]/(2\pi X \Z))^{\Z_m} \times (\R_0[[X, Y]])^{\Z_m^2}
  \end{equation*}
  that satisfies
  \begin{equation*} 
    \begin{cases}
      (\sfg_{u, v})^{(0, 1)} > 0, \\
      \sfs_u(X, Y) = \sfs_v(X, \sfg_{u, v}(X, Y)), \\
      \sfg_{u, u}(X, Y) = Y, \\
      \sfg_{u, w}(X, Y) = \sfg_{v, w}(X, \sfg_{u, v}(X, Y)),
    \end{cases}
  \end{equation*}
  for $u, v, w \in \Z_m$, under the action of $\Z_m$ by 
  \begin{equation*}
    [z] \cdot (\sfs_u, \sfg_{u, v})_{u, v \in \Z_m}  = (\sfs_{u + z}, \sfg_{u + z, v + z})_{u, v \in \Z_m}
  \end{equation*}
  for $z \in \Z$.
\end{definition}

\begin{definition} [{\cite[Definitions~3.6~and~3.7]{MR4694440}}] \label{def:Taylor-series-abs}
  The \emph{focus-focus label of $c$} is $\sfl^c = [\sfs^c_u, \sfg^c_{u, v}]_{u, v \in \Z_{m_c}}$ of multiplicity $m_c$.
\end{definition}

\begin{proposition}
  The Taylor series invariant at $c_i$ of multiplicity $m_i$ is a focus-focus label of multiplicity $m_i$.
\end{proposition}

\subsection{The relative twisting index}
\label{ssec:twist-index}

In semitoric systems, the twisting index invariant $k_i$ was originally defined in \cite{MR2534101} by comparing $A \circ \mu$ with a local preferred momentum map.
Since there is no preferred momentum map in the almost-toric case, we rather define the relative twisting index instead.

\begin{definition} \label{def:twisting-index}
  Let $\Pa{(M, \om, \mu), (c_i)_{i = 1}^\vf, (\epsilon_i)_{i = 1}^\vf}$ and $\Pa{(M', \om', \mu'), (c'_i)_{i = 1}^\vf, (\epsilon'_i)_{i = 1}^\vf}$ be two marked systems in $\tcalM_0$ with an isomorphism $G$ between their bases as marked almost-toric closed disks the same Taylor series labels.
  For $i \in \Set{1, \dotsc, \vf}$ and for any $u \in \Z_{m_i}$, let $\tsfs_{i, u}$ and $\tsfs'_{i, u}$ be the relative action Taylor series of the two marked systems respectively at $p_{i, u}$, with respect to a choice of local affine coordinate maps $A_i$ at $c_i$ and $A'_i$ at $c'_i$ with $A'_i \circ G = A_i$, and there is an integer $\sfk_i \in \Z$ such that $\tsfs'_{i, u} - \tsfs_{i, u} = 2k_i\pi X$ for any $u \in \Z_{m_i}$.
  We call $\sfk_i$ the \emph{twisting index} of $\Pa{(M', \om', \mu'), (c'_i)_{i = 1}^\vf, (\epsilon'_i)_{i = 1}^\vf}$ relative to $\Pa{(M, \om, \mu), (c_i)_{i = 1}^\vf, (\epsilon_i)_{i = 1}^\vf}$ at $c_i$.
\end{definition}

\begin{remark}
  Observe that, as in the case of semitoric systems in \cite{MR4797864}, the twisting index at $c_i$ is independent of the choice of $u \in \Z_{m_i}$.
  In principle, it is possible to use the relative action Taylor series instead in the Taylor series invariants so as to absorb the twisting indices as was done in \cite[Section~6]{MR4797864}; we choose not to do so until we find a better representation.
\end{remark}

\subsection{Classification in terms of base}
\label{sec:classification-base}

Initially, one wants to classify almost-toric systems, but there are symmetries on the bases unless we fix an ordering and a cut ray at each focus-focus value.

\begin{definition}
  A symplectic equivalence $G$ from a system $(M, \om, \mu)$ to $(M', \om', \mu')$ in $\tcalM$ is a \emph{symplectic equivalence} from a marked system $\Pa{(M, \om, \mu), (c_i)_{i = 1}^\vf, (\epsilon_i)_{i = 1}^\vf}$ to $\Pa{(M', \om', \mu'), (c'_i)_{i = 1}^\vf, (\epsilon'_i)_{i = 1}^\vf}$ in $\tcalM_0$ if $c'_i = G (c_i)$ and $\epsilon'_i = G(\epsilon_i)$ for each $i \in \Set{1, \dotsc, \vf}$.
  Let $\calM_0$ denote the collection of symplectic equivalence classes via orientation-preserving diffeomorphisms of marked systems in $\tcalM_0$.
\end{definition}

Let $\Pa{(M, \om, \mu), (c_i)_{i = 1}^\vf, (\epsilon_i)_{i = 1}^\vf} \in \tcalM_0$.
We assign to it an almost-toric closed disk, and to each marked point a focus-focus label and a twisting index.
These are packaged into a complete almost-toric ingredient.

\begin{definition} \label{def:Xhat}
  We denote by $\tcalX$ the set of tuples
  \begin{equation*}
    \Pa{(B, \calA, \calS), (c_i)_{i = 1}^\vf, (\epsilon_i)_{i = 1}^\vf, (\sfl_i)_{i = 1}^\vf, (\sfk_i)_{i = 1}^\vf}
  \end{equation*}
  such that
  \begin{itemize}
    \item $(B, \calA, \calS)$ is an almost-toric closed disk;
    \item $(c_i)_{i = 1}^\vf$ is an ordering of $B_\rmf$ with $m_i$ the multiplicity of $c_i$; 
    \item and for each $i \in \Set{1, \dotsc, \vf}$, $\epsilon_i$ is one of the two cut rays at $c_i$, 
    \item $\sfl_i$ is a focus-focus label \emph{compatible} with $(B, \calA, \calS)$, in the sense that $\sfl_i$ has multiplicity $m_i$,
    \item and $\sfk_i \in \Z$.
  \end{itemize}
  A \emph{complete almost-toric ingredient} is a class of elements of $\tcalX$ under the equivalence relation
  \begin{multline*}
    \Pa{(B, \calA, \calS), (c_i)_{i = 1}^\vf, (\epsilon_i)_{i = 1}^\vf, (\sfl_i)_{i = 1}^\vf, (\sfk_i)_{i = 1}^\vf} \\
    \sim \Pa{(B', \calA', \calS'), (G(c_i))_{i = 1}^\vf, (G(\epsilon_i))_{i = 1}^\vf, (\sfl_i)_{i = 1}^\vf, (\sfk_i)_{i = 1}^\vf}
  \end{multline*}
  where $G$ is an isomorphism of almost-toric closed disks from $(B, \calA, \calS)$ to $(B', \calA', \calS')$.
  We denote by $\calX$ the set of complete almost-toric ingredients.
\end{definition}

For any marked system in $\tcalM_0$ we assign to it an element of $\tcalX$:
\begin{align*}
  \trmi_0 \Pa{(M, \om, F), (c_i)_{i = 1}^\vf, (\epsilon_i)_{i = 1}^\vf} = \Pa{(B, \calA, \calS), (c_i)_{i = 1}^\vf, (\epsilon_i)_{i = 1}^\vf, (\sfl_i)_{i = 1}^\vf, (\sfk_i)_{i = 1}^\vf}
\end{align*}
where
\begin{itemize}
  \item $(B, \calA, \calS)$ is the base of $(M, \om, \mu)$;
  \item $\sfl_i$ is the Taylor series invariant at $c_i$, and 
  \item $\sfk_i \in \Z$ is the twisting index relative to the reference system.
\end{itemize}

Given any isomorphism class $\left[(B, \calA, \calS), (c_i)_{i = 1}^\vf, (\epsilon_i)_{i = 1}^\vf\right]$ of marked almost-toric closed disks and any tuple of compatible focus-focus labels $(\sfl_i)_{i = 1}^\vf$, a reference system, as mentioned above, is chosen in the subset of $\calM_0$ consisting of those with prescribed base and Taylor series invariants.
Then we assign to any marked system in this subset its twisting index $(\sfk_i)_{i = 1}^\vf \in \Z^\vf$ relative to this reference marked system.

\begin{lemma} \label{lem:atoric-classif-uniq}
  If $\calF, \calF' \in \tcalM_0$ are symplectically equivalent via an orientation-preserving $G$, then $\trmi_0 (\calF)$ is isomorphic to $\trmi_0 (\calF')$ via $G$ as elements of $\tcalX$.
\end{lemma}

\begin{proof}
  Let
  \begin{align*}
    \calF &= \Pa{(M, \om, \mu), (c_i)_{i = 1}^\vf, (\epsilon_i)_{i = 1}^\vf}, &\calF' &= \Pa{(M', \om', \mu'), (c'_i)_{i = 1}^\vf, (\epsilon'_i)_{i = 1}^\vf}.
  \end{align*}
  If $G$ takes $(M, \om, \mu)$ to $(M', \om' , \mu')$, then $G$ is an isomorphism from $(B, \calA, \calS)$ to $(B', \calA', \calS')$, and is already an isomorphism between marked almost-toric closed disks.
  We obtain $\sfl_i = \sfl'_i$ and $\sfk'_i = \sfk_i$ by \cref{def:Taylor-series} and their own definitions.
\end{proof}

\begin{corollary} \label{cor:atoric-classif-uniq}
  The map $\rmi_0$ given by
  \begin{align*}
    \rmi_0 \colon \calM_0  &\longrightarrow \calX, \\ 
    [\calF] &\longmapsto [\trmi_0 (\calF)]
  \end{align*}
  where $[\cdot]$ denote the isomorphism classes is well-defined.
\end{corollary}

\begin{lemma} \label{lem:atoric-classif-inj}
  If $\calF, \calF' \in \tcalM_0$, and $\trmi_0 (\calF)$ is isomorphic to $\trmi_0 (\calF')$ as elements of $\tcalX$, then $\calF$ and $\calF'$ are symplectically equivalent via an orientation-preserving map.
\end{lemma}

\begin{proof}
  Let
  \begin{align*}
    \calF &= \Pa{(M, \om, \mu), (c_i)_{i = 1}^\vf, (\epsilon_i)_{i = 1}^\vf}, &\calF' &= \Pa{(M', \om', \mu'), (c'_i)_{i = 1}^\vf, (\epsilon'_i)_{i = 1}^\vf}.
  \end{align*}
  Suppose
  \begin{align*}
    \trmi_0 (\calF) &= \Pa{(B, \calA, \calS), (c_i)_{i = 1}^\vf, (\epsilon_i)_{i = 1}^\vf, (\sfl_i)_{i = 1}^\vf, (\sfk_i)_{i = 1}^\vf} \\
    \trmi_0 (\calF') &= \Pa{(B', \calA', \calS'), (c'_i)_{i = 1}^\vf, (\epsilon'_i)_{i = 1}^\vf, (\sfl_i)_{i = 1}^\vf, (\sfk_i)_{i = 1}^\vf}
  \end{align*}
  are isomorphic via $G$.
  Then $G$ is an isomorphism from $(B, \calA, \calS)$ to $(B', \calA', \calS')$ with $c'_i = G (c_i)$ and $\epsilon'_i = G(\epsilon_i)$ for each $i \in \Set{1, \dotsc, \vf}$.
  Since the Taylor series invariants are equal, by \cite[Lemma~4.2]{MR4694440} there is an orientation-preserving $G_i$ from an open neighborhood $U_i$ of $c_i$ to $U'_i$ of $c'_i$, which lifts to a symplectomorphism $\varphi'_i \colon \mu^{-1} (U_i) \to (\mu')^{-1} (U'_i)$.
  In particular, for any $u \in \Z_{m_i}$ if $E_{i, u}$ is an orientation-preserving diffeomorphism about $c_i$ in Eliasson's linearization theorem then $E'_{i, u} = E_{i, u} \circ G_i^{-1}$ is one about $c'_i$.

  Moreover, $G_i^* \Res{\calA'}_{U'_i \setminus \epsilon'_i} = \Res{\calA}_{U_i \setminus \epsilon_i}$ for any $c \in U_i$.
  Compared with $G^* \calA' = \calA$ we deduce that $G_i^{-1} \circ G$ preserves $\Res{\calA}_{U_i \setminus \epsilon_i}$, and then there is an integer $d_i \in \Z$ such that $A_i \circ G_i^{-1} \circ G = T_i^{d_i} \circ A_i$ for any local affine coordinate map $A_i$ at $c_i$.
  Since $\sfk_i = \sfk'_i$, the relative twisting index of $(M', \om' , \mu')$ to $(M, \om, \mu)$ at $c_i$ vanishes.
  By the definition of the relative twisting index, with respect to $A_i$ and a local affine coordinate map $A'_i = A_i \circ G_i^{-1} \circ T_i^{-d_i}$ at $c'_i$, we have $\tsfs'_{i, u} = \tsfs_{i, u}$ for any $u \in \Z_{m_i}$.
  By \cref{eq:def-tS}, we deduce that
  \begin{align*}
    2\pi A_i^{\prime2} \circ G_i \circ E_{i, u}^{-1} - 2\pi A_i^2 \circ E_{i, u}^{-1} = \tldS'_i \circ (E'_{i, u})^{-1}  - \tldS_i \circ E_{i, u}^{-1}
   \end{align*}
  is a flat function which implies that $d_i = 0$.

  Hence, $A_i \circ G_i^{-1} \circ G - A_i$ is a flat function at $c_i$.
  Immitiating the proof of \cite[Proposition~4.2]{MR4797864} we could construct a symplectic equivalence by modifying $G$ so that it coincides with $G_i$ in each $U_i$, between the prescribed marked systems.
  The proof is complete.
\end{proof}

\begin{corollary} \label{cor:atoric-classif-inj}
  The map $\rmi_0$ is injective.
\end{corollary}

It remains to find the range of $\rmi_0$ inside of $\calX$.

\begin{lemma} \label{lem:atoric-classif-surj}
  For any $I \in \calX$, there is an $\calF \in \tcalM_0$ such that $\rmi_0 (\calF) = I$.
\end{lemma}

\begin{proof}
  Let
  \begin{align*}
    I = \left[(B, \calA, \calS), (c_i)_{i = 1}^\vf, (\epsilon_i)_{i = 1}^\vf, (\sfl_i)_{i = 1}^\vf, (\sfk_i)_{i = 1}^\vf\right].
  \end{align*}
  By \cite[Proposition 3.4]{MR2670163}, the almost-toric closed disk $(B, \calA, \calS)$ is the base of some initial almost-toric system $(M, \om, \mu)$ with markings $(c_i)_{i = 1}^\vf, (\epsilon_i)_{i = 1}^\vf$.
  It remains to replace $(\mu^{-1} (U_i), \om, \Res{\mu}_{\mu^{-1} (U_i)})$ where $U_i$ a neighborhood of $c_i$ in $B$, by the semiglobal model with Taylor series invariant $\sfl_i$ and twisting index $\sfk_i$.
  Since the semiglobal system $(\mu^{-1} (U_i), \om, \Res{\mu}_{\mu^{-1} (U_i)})$ is isomorphic to a semitoric one, we can use the proof of \cite[Proposition~4.7]{MR4797864} in the second stage (attaching focus-focus fibrations) and obtain a marked system $\calF' = \Pa{(M', \om', \mu'), (c'_i)_{i = 1}^\vf, (\epsilon'_i)_{i = 1}^\vf}$ so that $\rmi_0 (\calF) = I$.
\end{proof}

\begin{corollary} \label{cor:atoric-classif-surj}
  The map $\rmi_0$ is surjective.
\end{corollary}

We summarize our conclusion.

\begin{theorem}
  The fiber-preserving symplectomorphism classes of marked compact almost-toric systems in dimension four are in one-to-one correspondence with complete almost-toric ingredients by the map $\rmi_0$.
\end{theorem}

In order to obtain a classification of $\calM$, we use the forgetful map
\begin{align*}
  \pi \colon \tcalM_0 &\longrightarrow \tcalM, \\ 
  \Pa{(M, \om, \mu), (c_i)_{i = 1}^\vf, (\epsilon_i)_{i = 1}^\vf} &\longmapsto (M, \om, \mu)
\end{align*}
On the preimage of $(M, \om, \mu)$ there is an effective action by the group $S_\vf \times \Set{\pm1}^\vf$ as the permutation on $(c_i)_{i = 1}^\vf$ and the flipping on each $\epsilon_i$.
The corresponding partial action of $\frakG_\vf = S_\vf \times \Set{\pm1}^\vf$ on $\tcalX$ is
\begin{equation} \label{eq:action-tX} \begin{split}
  \rho &\Pa{(B, \calA, \calS), (c_i)_{i = 1}^\vf, (\epsilon_i)_{i = 1}^\vf, (\sfl_i)_{i = 1}^\vf, (\sfk_i)_{i = 1}^\vf} \\
  &= \Pa{(B, \calA, \calS), (c_{\rho(i)})_{i = 1}^\vf, (\epsilon_{\rho(i)})_{i = 1}^\vf, (\sfl_{\rho(i)})_{i = 1}^\vf, (\sfk_{\rho(i)})_{i = 1}^\vf},\text{ for }\rho \in S_\vf, \\
  \sigma &\Pa{(B, \calA, \calS), (c_i)_{i = 1}^\vf, (\epsilon_i)_{i = 1}^\vf, (\sfl_i)_{i = 1}^\vf, (\sfk_i)_{i = 1}^\vf} \\
  &= \Pa{(B, \calA, \calS), (c_i)_{i = 1}^\vf, (\sigma_i \epsilon_i)_{i = 1}^\vf, (\sigma_i \sfl_i)_{i = 1}^\vf, (\sigma_i \sfk_i)_{i = 1}^\vf}, \text{ for }\sigma \in \Set{\pm1}^\vf,
\end{split} \end{equation}
where any $1$ component of $\sigma$ gives the trivial action while
\begin{equation} \label{eq:action-tX-sigma} \begin{split}
  (-1) \ell_i^\pm &= \ell_i^\mp, \\
  (-1) [\sfs_u, \sfg_{u, v}]_{u, v \in \Z_m} &= [-\sfs_{-u} \circ \gamma, -\sfg_{-u, -v} \circ \gamma]_{u, v \in \Z_m}, \\
  (-1) \sfk_i &= -\sfk_i + \Delta_\sigma \sfk_i,
\end{split} \end{equation}
with $\gamma (X, Y) = (-X, -Y)$, and $\Delta_\sigma \sfk_i$ is the twisting index of the reference system in $\left[(B, \calA, \calS), (c_i)_{i = 1}^\vf, (\epsilon_i)_{i = 1}^\vf(\sfl_i)_{i = 1}^\vf\right]$ acted by $\sigma$ relative to the reference system in $\left[(B, \calA, \calS), (c_i)_{i = 1}^\vf, (\sigma_i \epsilon_i)_{i = 1}^\vf, (\sigma_i \sfl_i)_{i = 1}^\vf\right]$.
For justification of \cref{eq:action-tX-sigma} see \cite[Section~3.4]{MR4694440}.
We obtain the following classification.

\begin{corollary} \label{cor:atoric-classif}
  The map $\rmi$ given by
  \begin{align*}
    \rmi \colon \calM &\longrightarrow \calX / \sim_\frakG, \\ 
    [\calF] &\longmapsto [\trmi_0 (\calF)]
  \end{align*}
  is a bijection, where $\sim_\frakG$ is the equivalence relation induced by the partial action of $G_\vf$ for all $\vf \in \N$ defined by \cref{eq:action-tX,eq:action-tX-sigma}.
\end{corollary}


@preamble{ " \newcommand{\noop}[1]{} " }

@article {MR4694440,
    AUTHOR = {Pelayo, \'Alvaro and Tang, Xiudi},
     TITLE = {V\~u{} {N}g\d oc's conjecture on focus-focus singular fibers
              with multiple pinched points},
   JOURNAL = {J. Fixed Point Theory Appl.},
  FJOURNAL = {Journal of Fixed Point Theory and Applications},
    VOLUME = {26},
      YEAR = {2024},
    NUMBER = {1},
     PAGES = {Paper No. 6, 34},
      ISSN = {1661-7738,1661-7746},
   MRCLASS = {37J39 (53D20)},
  MRNUMBER = {4694440},
       DOI = {10.1007/s11784-023-01089-1},
       URL = {https://doi.org/10.1007/s11784-023-01089-1},
}

@article {MR4797864,
    AUTHOR = {Palmer, Joseph and Pelayo, \'Alvaro and Tang, Xiudi},
     TITLE = {Semitoric systems of non-simple type},
   JOURNAL = {Rev. R. Acad. Cienc. Exactas F\'is. Nat. Ser. A Mat. RACSAM},
  FJOURNAL = {Revista de la Real Academia de Ciencias Exactas, F\'isicas y
              Naturales. Serie A. Matematicas. RACSAM},
    VOLUME = {118},
      YEAR = {2024},
    NUMBER = {4},
     PAGES = {Paper No. 161},
      ISSN = {1578-7303,1579-1505},
   MRCLASS = {53D05 (37J35 53D20 53D35)},
  MRNUMBER = {4797864},
       DOI = {10.1007/s13398-024-01656-2},
       URL = {https://doi.org/10.1007/s13398-024-01656-2},
}

@phdthesis {Eliasson_1984,
     place = {Stockholm},
     title = {Hamiltonian systems with Poisson commuting integrals},
      ISBN = {91-7146-281-3},
       url = {http://urn.kb.se/resolve?urn=urn:nbn:se:su:diva-174707},
      note = {Diss. Stockholm : Univ.},
    author = {Eliasson, Håkan},
      year = {1984},
}

@article {MR642416,
    AUTHOR = {Atiyah, M. F.},
     TITLE = {Convexity and commuting {H}amiltonians},
   JOURNAL = {Bull. London Math. Soc.},
  FJOURNAL = {The Bulletin of the London Mathematical Society},
    VOLUME = {14},
      YEAR = {1982},
    NUMBER = {1},
     PAGES = {1--15},
      ISSN = {0024-6093},
   MRCLASS = {53C15 (22E30 53C55 58F05)},
  MRNUMBER = {642416},
MRREVIEWER = {A. Morimoto},
       URL = {https://doi.org/10.1112/blms/14.1.1},
}

@article {MR664117,
    AUTHOR = {Guillemin, V. and Sternberg, S.},
     TITLE = {Convexity properties of the moment mapping},
   JOURNAL = {Invent. Math.},
  FJOURNAL = {Inventiones Mathematicae},
    VOLUME = {67},
      YEAR = {1982},
    NUMBER = {3},
     PAGES = {491--513},
      ISSN = {0020-9910},
   MRCLASS = {58F05 (57S25 70H05)},
  MRNUMBER = {664117},
MRREVIEWER = {I. Vaisman},
       URL = {https://doi.org/10.1007/BF01398933},
}

@article {MR984900,
    AUTHOR = {Delzant, Thomas},
     TITLE = {Hamiltoniens p\'{e}riodiques et images convexes de l'application
              moment},
   JOURNAL = {Bull. Soc. Math. France},
  FJOURNAL = {Bulletin de la Soci\'{e}t\'{e} Math\'{e}matique de France},
    VOLUME = {116},
      YEAR = {1988},
    NUMBER = {3},
     PAGES = {315--339},
      ISSN = {0037-9484},
   MRCLASS = {58F05},
  MRNUMBER = {984900},
MRREVIEWER = {J. J. Duistermaat},
       URL = {http://www.numdam.org/item?id=BSMF_1988__116_3_315_0},
}

@incollection {MR2024634,
    AUTHOR = {Symington, Margaret},
     TITLE = {Four dimensions from two in symplectic topology},
 BOOKTITLE = {Topology and geometry of manifolds ({A}thens, {GA}, 2001)},
    SERIES = {Proc. Sympos. Pure Math.},
    VOLUME = {71},
     PAGES = {153--208},
 PUBLISHER = {Amer. Math. Soc., Providence, RI},
      YEAR = {2003},
   MRCLASS = {53D35 (53D20 55R55 57R17)},
  MRNUMBER = {2024634},
MRREVIEWER = {Vicente Mu\~{n}oz},
       DOI = {10.1090/pspum/071/2024634},
       URL = {https://doi-org.myaccess.library.utoronto.ca/10.1090/pspum/071/2024634},
}

@Article{MR2018823,
  author           = {Zung, Nguyen Tien},
  date             = {2003},
  journaltitle     = {Compositio Mathematica},
  title            = {{Symplectic topology of integrable Hamiltonian systems. II. Topological classification}},
  doi              = {10.1023/A:1026133814607},
  issn             = {0010-437X},
  number           = {2},
  pages            = {125--156},
  volume           = {138},
}

@article {MR2534101,
    AUTHOR = {Pelayo, {\'A}lvaro and V{\~u} Ng{\d{o}}c, San},
     TITLE = {Semitoric integrable systems on symplectic 4-manifolds},
   JOURNAL = {Invent. Math.},
  FJOURNAL = {Inventiones Mathematicae},
    VOLUME = {177},
      YEAR = {2009},
    NUMBER = {3},
     PAGES = {571--597},
      ISSN = {0020-9910},
   MRCLASS = {37J35 (37J05 37J15 53D35)},
  MRNUMBER = {2534101},
MRREVIEWER = {Martin Pinsonnault},
       URL = {https://doi.org/10.1007/s00222-009-0190-x},
}

@article {MR2784664,
    AUTHOR = {Pelayo, \'Alvaro and V{\~u} Ng{\d{o}}c, San},
     TITLE = {Constructing integrable systems of semitoric type},
   JOURNAL = {Acta Math.},
  FJOURNAL = {Acta Mathematica},
    VOLUME = {206},
      YEAR = {2011},
    NUMBER = {1},
     PAGES = {93--125},
      ISSN = {0001-5962},
   MRCLASS = {53D20 (37J15 37J35)},
  MRNUMBER = {2784664},
MRREVIEWER = {Diego Matessi},
       URL = {https://doi.org/10.1007/s11511-011-0060-4},
}

@Article{MR2670163,
  author           = {Leung, Naichung Conan and Symington, Margaret},
  date             = {2010},
  journaltitle     = {Journal of Symplectic Geometry},
  title            = {{Almost toric symplectic four-manifolds}},
  doi              = {10.4310/JSG.2010.v8.n2.a2},
  issn             = {15275256},
  number           = {2},
  pages            = {143--187},
  volume           = {8},
}

@article {MR3098203,
    AUTHOR = {V\~{u} Ng\d{o}c, San and Wacheux, Christophe},
     TITLE = {Smooth normal forms for integrable {H}amiltonian systems near
              a focus-focus singularity},
   JOURNAL = {Acta Math. Vietnam.},
  FJOURNAL = {Acta Mathematica Vietnamica},
    VOLUME = {38},
      YEAR = {2013},
    NUMBER = {1},
     PAGES = {107--122},
      ISSN = {0251-4184},
   MRCLASS = {37J05 (53D12 53D20 70G60 70H33)},
  MRNUMBER = {3098203},
MRREVIEWER = {Mircea Crasmareanu},
       DOI = {10.1007/s40306-013-0012-5},
       URL = {https://doi.org/10.1007/s40306-013-0012-5},
}

@Article{MR3371718,
  author           = {Karshon, Yael and Lerman, Eugene M.},
  date             = {2015-07},
  journaltitle     = {SIGMA. Symmetry, Integrability and Geometry: Methods and Applications},
  title            = {{Non-compact symplectic toric manifolds}},
  doi              = {10.3842/sigma.2015.055},
  issn             = {18150659},
  pages            = {055},
  url              = {http://www.emis.de/journals/SIGMA/2015/055/},
  volume           = {11},
  arxivid          = {0907.2891},
  modificationdate = {2023-09-01T20:02:15},
  publisher        = {SIGMA. Symmetry, Integrability and Geometry: Methods and Applications},
}
\end{document}